\documentclass{amsart}
\usepackage{amsmath,amssymb,hyperref,mathrsfs,graphicx,stmaryrd}
\usepackage[utf8x]{inputenc}
\usepackage{amsfonts}
\usepackage{latexsym}
\usepackage{amsthm}
\usepackage{amssymb,amscd}
\usepackage{xargs}
\usepackage{color}

\usepackage{todonotes}

\usepackage{microtype}

\usepackage{textcomp}
\usepackage{lmodern}
\usepackage[text={5.5in, 9in}, centering]{geometry}
\newtheorem{thm}{Theorem}[section]

\newtheorem{lem}[thm]{Lemma}
\newtheorem{defn}[thm]{Definition}
\newtheorem{que}[thm]{Question}

\newtheorem{prop}[thm]{Proposition}
\newtheorem{cor}[thm]{Corollary}
\newtheorem{rmk}[thm]{Remark}
\newcommand{\be}{\begin{eqnarray}}
\newcommand{\ee}{\end{eqnarray}}
\newcommand{\beal}{\begin{aligned}}
\newcommand{\enal}{\end{aligned}}
\newcommand{\bc}{\begin{cases}}
\newcommand{\ec}{\end{cases}}

\newcommand{\Z}{\mathbb{Z}}

\newcommand{\R}{\mathbb{R}}
\newcommand{\Q}{\mathbb{Q}}
\newcommand{\A}{\mathbb{A}}

\newcommand{\N}{\mathbb{N}}
\newcommand{\cU}{\mathcal{U}}

\newcommand{\cN}{\mathcal{N}}
\newcommand{\cI}{\mathcal{I}}

\newcommand{\cT}{\mathcal{T}}
\newcommand{\cA}{\mathcal{A}}

\newcommand{\cM}{\mathcal{M}}

\newcommand{\cV}{\mathcal{V}}

\newcommand{\eps}{\epsilon}
\newcommand{\wt}{\widetilde}
\newcommand{\wh}{\widehat}
\newcommand{\ol}{\overline}

\newcommand{\om}{\omega}

\newcommand{\Lb}{\Lambda}
\newcommand{\dt}{\delta}
\newcommand{\Dt}{\Delta}

\newcommand{\T}{\mathbb{T}}

\title[Global Behaviors of weak KAM solutions for exact symplectic Twist Maps]{Global Behaviors of weak KAM Solutions for exact symplectic Twist Maps}
\author{Jianlu Zhang}
\address{Hua Loo-Keng Key Laboratory of Mathematics \& Mathematics Institute, Academy of Mathematics and systems science, Chinese Academy of Sciences, Beijing 100190, China}
\thanks{}
\email{jellychung1987@gmail.com}
\date{}
\begin{document}
\maketitle

\begin{abstract}
We investigated several global behaviors of the weak KAM solutions $u_c(x,t)$ parametrized by $c\in H^1(\T,\R)$. For the suspended Hamiltonian $H(x,p,t)$ of the exact symplectic twist map, we could find a family of weak KAM solutions $u_c(x,t)$ parametrized by $c(\sigma)\in H^1(\T,\R)$ with $c(\sigma)$ continuous and monotonic and
\[
\partial_tu_c+H(x,\partial_x u_c+c,t)=\alpha(c),\quad \text{a.e.\ } (x,t)\in\T^2,
\]
such that sequence of weak KAM solutions $\{u_c\}_{c\in H^1(\T,\R)}$ is $1/2-$H\"older continuity of parameter $\sigma\in \mathbb{R}$. Moreover, for each generalized characteristic (no matter regular or singular) solving 
\[
\left\{
\begin{aligned}
&\dot{x}(s)\in \text{co} \Big[\partial_pH\Big(x(s),c+D^+u_c\big(x(s),s+t\big),s+t\Big)\Big], & \\
&x(0)=x_0,\quad (x_0,t)\in\T^2,& 
\end{aligned}
\right.
\]
we evaluate it by a uniquely identified rotational number $\om(c)\in H_1(\T,\R)$. This property leads to a certain topological obstruction in the phase space and causes local transitive phenomenon of trajectories. Besides, we discussed this applies to high-dimensional cases.
\end{abstract}
\vspace{20pt}

\tableofcontents

\section{Introduction}\label{s1}
\vspace{20pt}

The earliest survey of the area preserving maps can be found from Poincar\'e's research on the three-body problem \cite{P}, which firstly revealed the chaotic phenomenon of low dimensional dynamics. After that, Birkhoff made a systematic research of the area preserving map defined in an annulus region\cite{B}, which inspired the development of other related topics, e.g. the convex billiard map, the geodesic flow on surfaces, etc \cite{Z1,Z2,Ba}. Results of these topics gradually extended the territory of the area preserving maps and comprised the low dimensional dynamic theory of the second half of $20^{th}$ century \cite{G}. Especially, the use of variational method greatly boosted the theoretical development, due to the work of Mather in 1980's. That leads to a flurry of finding invariant sets parametrized by certain {\sf rotational numbers}, both in mathematics and physics.\\

As a direct offspring of these research, the research of {\sf exact symplectic twist maps} is still meaningful and enlightening to the exploration of high dimensional dynamics in nowadays. We can now formalize it by the following $C^1$ diffeomorphism 
\be 
f:(x,p)\rightarrow(x', p'),\quad (x,p)\in T^*\T
\ee
of which we denote by $\wt f(\ol x,p)$ the lift
of this map to the universal covering space, i.e.
 \[
 \wt f(\ol x,p)=(x'+m,p')
 \]
  for all $(\ol x,p)\in T^*\R$ satisfying $\ol x=x+m$, $m\in\Z$. These properties hold for the map:

 \begin{itemize}
 \item {\sf (area-preserving)}
 $dx\wedge dp=dx'\wedge dp'$. 
 
  \item {\sf (exact)} for any noncontractible curve $\gamma\in T^*\T$, 
  \[
  \int_{\gamma}f^*p dx- p dx=0.
  \]
  
 \item {\sf (twist)} $\partial p'/\partial x>0$, equivalently the image 
 $f(L_{x_0})$ of every vertical line $L_{x_0}=\{(x,p)\in T^*\T|x=x_0\}$  
 is monotonically twisted in the $x$-component.  
 \end{itemize}
 The $f$ satisfying the first two items is called a {\sf twist map}
 and satisfying all three items is called an {\sf exact symplectic twist map}. In \cite{M2}, Moser successfully suspended this kind of map into a time-periodic Hamiltonian flow:
 \begin{thm}
For any $C^1$ exact symplectic twist map $f$ defined on bounded annulus $\A:=\{x\in\T,p\in[a,b]\}$,  there exists a time periodic Hamiltonian $H(x,p,t) \in C^2(\A \times \T)$ with positively definite Hessian matrix $\partial^2H/\partial p^2$, such that $f$ coincides with the time-1 map $\phi_H^1$.
\end{thm}
The significance of this result is that it connects the dynamics of exact symplectic twist map $f$ to the variational properties of the corresponding Lagrangian (now is known as the {\sf Tonelli Lagrangian})
\[
L(x,v,t):=\max_{p\in T_x^*M}\{\langle v,p\rangle-H(x,p,t)\},\quad\text{for\ }(x,v,t)\in T\T\times\T
\]
 Therefore, we can find variational minimal orbits with different topological properties, which form different invariant sets in the phase space. That's the essence of the {\sf Aubry Mather theory}, see \cite{Mat1,Mat2}.\\

Based on previous Theorem, we can now propose a Tonelli Lagrangian $L(x,v,t)$ satisfying the {\sf Standing Assumptions}:
\begin{itemize}
\item {\sf (Smoothness)} $L(x,v,t)$ is $C^2$ smooth of $(x,v,t)\in TM\times\T$;
\item {\sf (Positive Definiteness)} the Hessian matrix $L_{vv}$ is positively definite for any $(x,v,t)\in TM\times\T$;
\item {\sf (Completeness)} the Euler-Lagrange equation of $L(x,v,t)$ is well defined for the whole time $t\in\R$;
\end{itemize}
 where $M$ is any smooth, boundless compact manifold (in the current paper $M=\T$).\\
 
We need to specify that, as a parallel correspondence of the Aubry Mather theory, Fathi developed a  PDE viewpoint in the early of the $21^{st}$ century \cite{F}. Precisely, we could find a list of so called {\sf weak KAM solutions} of the following {\sf Static Evolutionary Hamilton Jacobi equation}:
\be\label{eq:ehj}
\partial_tu_c+H(x,\partial_x u_c+c,t)=\alpha(c),\quad\forall c\in H^1(\T,\R),\; \text{a.e.\ } (x,t)\in\T^2.
\ee
For every fixed $t\in\T$, $u_c(x,t)$ is a {\sf semiconcave function} of $x\in\T$ with linear module \cite{Ca}. For any $(x,t)\in\T^2$, the {\sf super differential set} $c+D_x^+u_c(x,t)$ is a convex set of $\R$.  If $p\in ex\big(c+D_x^+u_c(x,t)\big)$ is an extremal point, then $(x,t,p,\partial_t u_c(x,t))$ will decide a unique {\sf backward semi-static orbit} as the initial point (see Sec.\ref{s2} for the proof). More conclusions about the weak KAM solutions can be found in Sec. \ref{s2} with details.\\

As a warmup, we now exhibit a dynamic simulation of the {\sf standard map}, to give  the readers a concrete impression of the global behaviors which the parametrized weak KAM solutions $u_c$ may possess: Let's start from a integrable map $f_0:(x,p)\rightarrow(x+p,p)$, of which we can see that the whole phase space is foliated by invariant circles $\cT_w=\{(x,\om)\in T^*\T|\om\in\R\}$. That implies we can find a list of trivial weak KAM solutions of (\ref{eq:ehj}) satisfying $u_c(x,t)\equiv0, c=w$ and equation (\ref{eq:gc}) becomes trivial $\dot{x}(s)=w$. If we perturb $f_0$ by $f_\eps= f_0+\eps f_1$
\[
f_\eps(x, p)=(x+p+\eps V(x),p+\eps V(x)),\quad0<\eps<1
\]
and gradually increase $\eps$, we could  observe that at first most of the tori $\cT_w$ preserve and just deform a little bit (the KAM theorem ensures), then gradually they break up and turn into a chaotic state, as shown in Fig. \ref{eps-sm}. Accordingly, as $\eps$ raises, for more and more $c\in H^1(\T,\R)$ the associated $u_c$ will lose the $C^1$ smoothness and singularity will come out and propagate.\\
\begin{figure}
\begin{center}
\includegraphics[width=9cm]{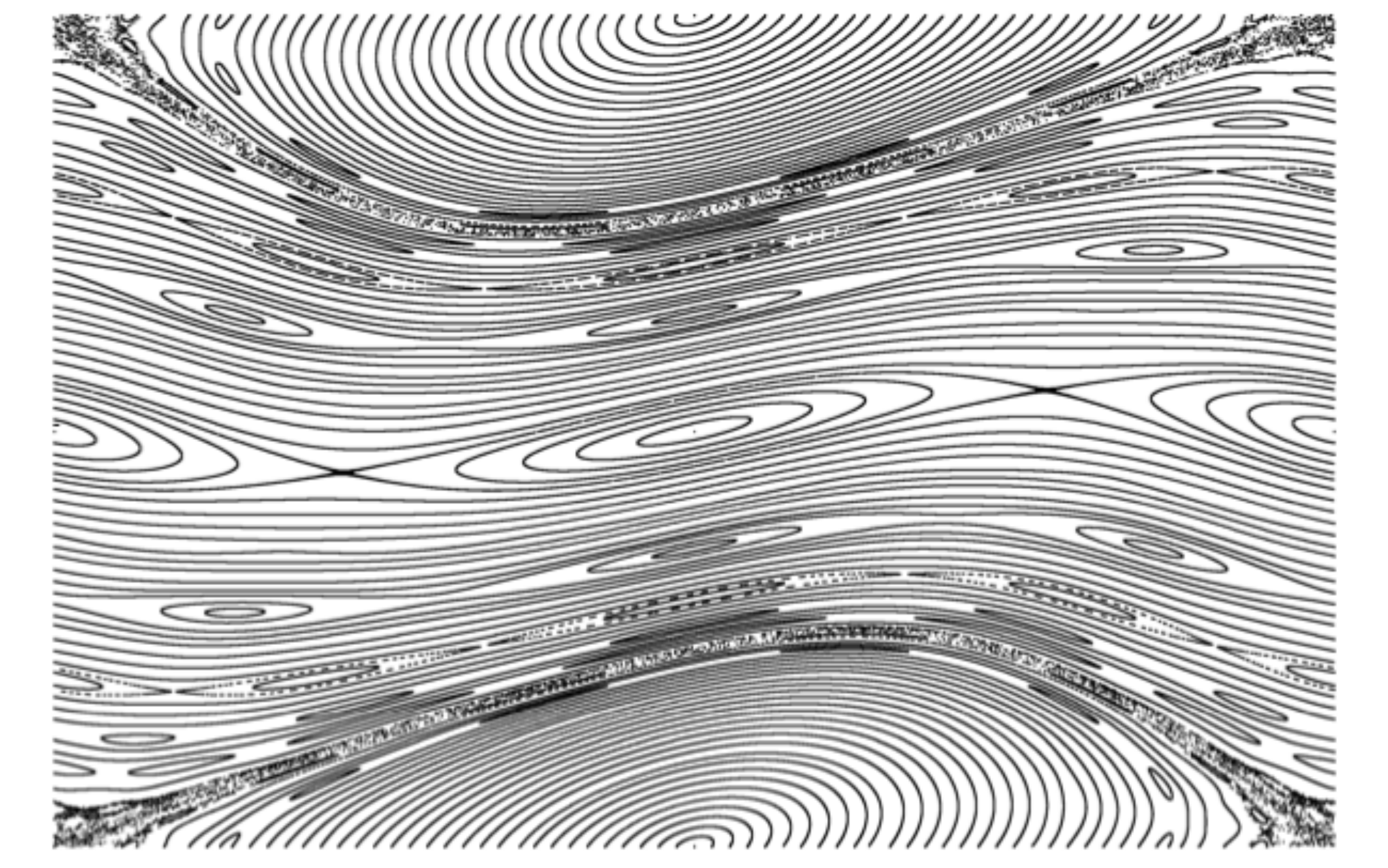}
\caption{Phase portrait of the stand map for $V(x)=\cos2\pi x$, $\eps=2\pi\cdot 0.1$}
\label{eps-sm}
\end{center}
\end{figure}

For a general symplectic twist map, previous process is still observable. However, we do expect to `pick up' enough trajectories in the phase space, to persist the global foliation structure, which might has a weak regularity:

\begin{thm}[{\sf Regularity}]\label{thm:holder}
There exists a sequence of weak KAM solutions $\{u_{c(\sigma)}(x,t)+c(\sigma)x|x\in\R,t\in\T\}_{\sigma\in\R}$ of (\ref{eq:ehj}), which is $1/2-$H\"older  continuous w.r.t. the parameter $\sigma\in \R$. Here $c(\sigma)\in H^1(\T,\R)$ is a strictly increasing continuous function.
\end{thm}
\begin{rmk}
This conclusion was initially proposed by Mather in a sketch of {\it a priori} unstable Arnold diffusion problem \cite{Mat3}, to construct a global transition chain benefiting from a {\sf normally hyperbolic invariant cylinder} structure. The dynamic on the cylinder is exactly decided by a symplectic twist map and the regularity of the weak KAM solutions w.r.t. some effective parameter will lead to the regularity of the stable (resp. unstable) manifold of the cylinder. Later it was proved in \cite{Z} for generic twist maps (with a hyperbolicity assumption). Here we remove the `genericity' condition in \cite{Z} and verified that the global regularity of weak KAM solutions exists for general exact symplectic twist maps. Besides, a two dimensional Finsler metric case is considered in \cite{CX}, where they define an `elementary weak KAM solution' to avoid the analysis of  $c\in H^1(\T^2,\R)$.

 A heuristic understanding of this Theorem is that although the global foliation structure of invariant tori $\cT_w$ may not exist for general twist maps, a `weak foliation' structure consisting of {\sf backward invariant tori} 
 \[
\cT_{\om(c)}:= \Big\{\Big(x,t,ex\big(D^+u_c(x,t)+c\big),\partial_tu_c(x,t)\Big)\subset T^*\T^2\Big|(x,t)\in\T^2\Big\}_{c\in H^1(\T,\R)}
 \]
  could still be found.
\end{rmk} 

Recall that the lack of regularity of previous weak foliation is essentially caused by the singularity of the weak KAM solutions $\{u_c\}_{c\in H^1(\T,\R)}$. Although the singular points of each $u_c$ just form a measure zero set in the configuration space $\T^2$, they indeed changes the topological structure of the phase space and complex dynamic phenomena happen \cite{Mat3}.\\

Nonetheless, the propagation of the singularity is still predictable. Usually we introduce the following differential inclusion equation:
\be\label{eq:gc}
\left\{
\begin{aligned}
&\dot x:=\frac{dx}{ds}\in \text{co}\Big[ \partial_pH\Big(x(s),c+D^+u_c\big(x(s),s+t\big),s+t\Big)\Big],\\
&x(0)=x_0\in\T,\;t\in\T.
\end{aligned}
\right.
\ee
we can see that any solution of (\ref{eq:gc}) is unique for fixed initial point, and such a solution is called a {\sf generalized characteristic} (GC for short). Whatever $(x,t)\in\T^2$ is {\sf singular} or {\sf regular} of $u_c$, this definition always ensures the existence of the GC starting from it. Moreover, the propagation of the GCs has the following property:


\begin{thm}[{\sf Rotation Number}]\label{thm:g-c-t}
$\forall (x,t)\in\T^2$ as the initial state, the solution $\gamma(s)$ of (\ref{eq:gc}) is unique and at least one-side unbounded, namely it's well defined for at least one of $[0,+\infty)$ and $(-\infty,0]$. Moreover, the {\it rotation number $\rho(\gamma)$} defined by 
\[
\rho(\gamma)=\lim_{\substack{T\rightarrow+\infty\\(\text{resp.} -\infty)}}\frac{\gamma(T+s)-\gamma(s)}{T},\quad\forall s\in\T
\] 
equals $\alpha'(c)$, which is the first derivative of Mather's $\alpha-$function.
\end{thm}
\begin{rmk}
Due to previous Theorem \ref{thm:g-c-t},  for $\alpha'(c)=\R\backslash\Q$, each singular GC of $\alpha'(c)$ rotation number has to be asymptotic to the {\sf projected Aubry set} $\cA(c)$. 

As for the case $\alpha'(c)=p/q\in\Q$, the topological structure of the singular GCs would be much more complicated. Notice that each GC $\gamma$ has no self-intersection, if we lift them into the universal space $\{(x,t)\in\R^2\}$, the constraint of dimension will decide 3 different types by the following:
\begin{itemize}
\item $\Big(\dfrac pq\Big)^+-$type, if the lift $\ol\gamma$ satisfies $\ol\gamma(q)-\ol\gamma(0)>p$;
\item $\Big(\dfrac pq\Big)^--$type, if the lift $\ol\gamma$ satisfies $\ol\gamma(q)-\ol\gamma(0)<p$;
\item $\dfrac pq-$type, if the lift $\ol\gamma$ satisfies $\ol\gamma(q)-\ol\gamma(0)=p$.
\end{itemize}
These rotation symbols were firstly introduced by Mather in \cite{Mat0}. Benefit from these, we can get a clearer classification of the singular GCs now:
%
%
 \begin{enumerate}
 \item periodic \Big[$p/q-$type\Big];
  \item $\Big(\dfrac pq\Big)^+-$asymptotic to $\cA(c)$;
\item $\Big(\dfrac pq\Big)^--$asymptotic to $\cA(c)$.
 \item $\Big(\dfrac pq\Big)^+-$asymptotic to case (1);
 \item $\Big(\dfrac pq\Big)^--$asymptotic to case (1).
  \end{enumerate}
 The former 3 types are common in the phase space. For instance, for the stand map with $V(x)=\cos2\pi x$ (mentioned before), there exists an interval $[-c_0,c_0]$ such that for any $c$ in it, $\alpha'(c)=0$. Then for suitably small $\eps\ll1$, we can find (1)-type singular GC of $c=0$, (2)-type singular GC for $c=c_0$ and (3)-type singular GC for $c=-c_0$.
 
 However, we confess that we couldn't exclude the existence of the later 2 types of singular GCs. It could be artificially constructed for some maps with sort of `fragile dynamics', but shouldn't be typical.
\end{rmk}

The last fact we would like to illustrate, is that the singularity would never happen for isolated $c\in H^1(\T,\R)$. Precisely, for those $c\in H^1(\T,\R)$ of which (\ref{eq:ehj}) inherits no classical solutions, the set they form can be denoted by $\mathfrak I\subset H^1(\T,\R)$; If we take the interior of $\mathfrak I$, then
\[
\mathring{\mathfrak I}=\bigcup_{i=1}^\infty(a_i,b_i)
\]
of which each open interval $(a_i,b_i)$ ({\sf Instability Interval}) corresponds to a so called {\sf Birkhoff Instability Region} (BIS for short) in the phase space. The existence of wandering orbits in the BIS is proved by Mather:

\begin{thm}[Mather \cite{Mat}]\label{thm:s-t}
For any $c,c'\in (a_i,b_i)$, $\wt\cA(c)$ and $\wt\cA(c')$ are dynamically connected, namely, there exists heteroclinic orbits connecting them.
\end{thm}

The original proof in \cite{Mat} of this result is rather complicated. Here we gave a simplified proof in Sec. \ref{s5}, by making use of the global properties which have been proved in aformentioned theorems. Besides, we proposed several heuristic remarks in Sec. \ref{s5}, to show the possibility of a generalization to high dimensional case.\\

\subsection{Organization of the article} This paper is organized as follows: In Sec. \ref{s2} we reviewed some background knowledge of the weak KAM solutions and generalized characteristics for twist maps. Based on these results, we gave the proof of Theorem \ref{thm:holder} in Sec. \ref{s3}; In Sec. \ref{s4}, we proved Theorem \ref{thm:g-c-t}. Finally, in Sec. \ref{s5} we proved Theorem \ref{thm:s-t} and gave a summary of possible extensions.\\

\noindent\textbf{Acknowledgement} This work is supported by the Natural Scientific Foundation of China E0110002 (Grant No. 11901560). The author is grateful to Prof. Arnaud for introducing their result on the $c-$continuity of weak KAM solutions \cite{AZ}, which is hopeful to be used to prove the H\"older-continuity w.r.t. $c$ in the furture work. The author also thanks the anonymous referee for helpful revisions and suggestions. 

\vspace{20pt}

\section{Preliminary: Mather Theory and the weak KAM solutions of time-periodic Lagrangians}\label{s2}

\vspace{20pt}

\subsection{Mather Theory for Tonelli Lagrangians} For the time-periodic Lagrangian $L(x,v,t)$ satisfying our standing assumptions (with general manifold $M$), the {\sf critical curve} is usually defined by $\gamma :\R\rightarrow M$, such that the following {\sf Euler-Lagrange equation} holds
 \be\label{eq:el}
 \frac {d}{dt}L_v(\gamma,\dot\gamma,t)=L_x(\gamma,\dot\gamma,t)
 \ee
for all $t\in\R$. Notice that the minimizer of the following 
\be\label{eq:act-zero}
h((x,t),(y,t'))&:=&\min_{\substack{\gamma\in C^{ac}([t,t'],M)\\\gamma(t)=x,\gamma(t')=y}}A(\gamma)\big|_{[t,t']}\nonumber\\
&=&\min_{\substack{\gamma\in C^{ac}([t,t'],M)\\\gamma(t)=x,\gamma(t')=y}}\int_t^{t'}L(\gamma,\dot\gamma,\tau)d\tau
\ee
has to be a solution of the Euler-Lagrange equation on $[t,t']\subset\R$. On the other side, due to the completeness assumption, for any $(x,v,s)\in TM\times\T$, there exists a unique critical curve $\gamma$  starting from it, and can be extended for all $t\in\R$. If we denote by $\phi_L^t$ the Euler Lagrange flow, we can make use of the {\sf Birkhoff Ergodic Theorem} and get a $\phi_L^t-$invariant probability measure by
\be
\int_{TM\times\T}fd\mu_\gamma:=\lim_{T\rightarrow+\infty}\frac1T\int_0^Tf(\gamma,\dot\gamma,t)dt,\quad\forall f\in C_c(TM\times\T,\R).
\ee
Gather all these invariant probability measure into a set $\mathfrak M_L$, for any closed $1-$form $\eta(x)dx$ with $[\eta]=c\in H^1(M,\R)$, the parametrized Lagrangian 
 $L_c(x,v,t):=L(x,v,t)-\langle c,v\rangle$ possesses the same Euler-Lagrange equation with $L(x,v,t)$, so the following {\sf Mather's Alpha function} $\alpha:H^1(M,\R)\rightarrow\R$
\be\label{eq:alpha}
\alpha(c)=-\min_{\mu\in\mathfrak{M}_L}\int L-\eta\; d\mu, \quad[\eta]=c\in H^1(M,\R)
\ee
is well defined and the minimizers form a set $\mathfrak M(c)$, which is contained in $\mathfrak M_L$. The {\sf Mather set} is defined by 
\[
\wt\cM(c):=\bigcup_{\mu\in\mathfrak M(c)}supp(\mu),
\]
which is graphic in the phase space:
\begin{thm}[Graphic \cite{Mat2}]
$\wt\cM(c)$ is a Lipschitz graph over the {\sf projected Mather set} $\cM(c):=\pi_x\wt\cM(c)$, i.e.
\[
\pi_x^{-1}:\cM(c)\rightarrow TM\times\T
\]
is a Lipschitz function.
\end{thm}
As the conjugation of $\alpha(c)$, we can define the {\sf Mather's Beta function} $\beta: H_1(M,\R)\rightarrow\R$ by
\be
\beta(h)=\inf_{\substack{\mu\in\mathfrak{M}_L, \\\rho(\mu)=h}}\int L\;d\mu
\ee
where $\rho(\mu)\in H_1(M,\R)$ is defined by 
\[
\langle [\lambda], \rho(\mu)\rangle:=\int\lambda\;d\mu,\quad\forall \text{\;closed 1-form \;}\lambda \text{\;on\;} M.
\]
Due to the positive definiteness assumption, both $\alpha(c)$ and $\beta(h)$ are convex and superlinear. Besides, 
\be\label{eq:ineq-a-b}
\langle c,h\rangle\leq\alpha(c)+\beta(h),\quad\forall c\in H^1(M,\R),\; h\in H_1(M,\R),
\ee
of which the equality holds only for $c\in D^-\beta(h)$ and $h\in D^-\alpha(c)$, namely $c$ is contained in the {\sf sub derivative set} of $\beta(h)$ and $h$ is contained in the sub derivative set of $\alpha(c)$.\\

Follow the setting of \cite{B}, for any $\gamma\in C^{ac}([t,t'],M)$ we have
\begin{equation}
A_c(\gamma)\big{|}_{[t,t']}=\int_t^{t'}L(\gamma(t),\dot{\gamma}(t),t)-\langle\eta_c(\gamma(t)),\dot{\gamma}(t)\rangle dt+\alpha(c)(t'-t)
\end{equation}
and the {\sf $c-$action function} 
\begin{equation}
h_c((x,t),(y,t'))=\inf_{\substack{\xi\in C^{ac}([t,t'],M)\\
\xi(t)=x\\
\xi(t')=y}}A_c(\xi)\big{|}_{[t,t']},
\end{equation}
where $t,t'\in\mathbb{R}$ with $t<t'$. Therefore, the {\sf Ma\~n\'e Potential function}
\begin{equation}
F_c((x,\tau),(y,\tau'))=\inf_{\substack{[t]=\tau\\
[t']=\tau'\\t\leq t'}}h_c((x,t),(y,t')),
\end{equation}
is well defined on $M\times\T\times M\times \T$. Here for any $t\in\R$, $[t]\in[0,1)$ is uniquely identified by $t\equiv [t]\ (mod\ 1)$. Furthermore, the {\sf Peierl's barrier function}
\be
h_c^\infty((x,\tau),(y,\tau'))=\liminf_{t'-t\rightarrow+\infty}h_c((x,t),(y,t'))
\ee
is also well defined, where $[t]=\tau,[t']=\tau'$.

\begin{defn}
An absolutely continuous curve $\gamma:\mathbb{R}\rightarrow M$ is called a {\sf c-semistatic} if 
\[
F_c((x,\tau),(y,\tau'))=A_c(\gamma)\big{|}_{[t,t']},
\]
for all $t,t'\in\mathbb{R}$ satisfying $[t]=\tau$ and $[t']=\tau'$. A $c-$semi static curve $\gamma$ is called {\sf c-static} if
\[
A_c(\gamma)\big{|}_{[t,t']}+F_c((\gamma(t'),[t']),(\gamma(t),[t]))=0,\quad\forall t<t'\in\mathbb{R}.
\]
\end{defn}
The {\sf Ma\~{n}\'e set} $\widetilde{\mathcal{N}}(c)\subset TM\times\T$ is defined by the set of all the c-semi static orbits, and the {\sf Aubry set} $\tilde{\mathcal{A}}(c)$ is the set of all the c-static orbits. 
\begin{thm}[Graphic II \cite{B}]
$\wt\cA(c)$ is a Lipschitz graph over the {\sf projected Aubry set} $\cA(c):=\pi_x\wt\cA(c)$, i.e.
\[
\pi_x^{-1}:\cA(c)\rightarrow TM\times\T
\]
is a Lipschitz function.
\end{thm}
\begin{thm}\cite{Mat3}
$(x,\tau)\in\cA(c)$ iff $h^\infty_c((x,\tau),(x,\tau))=0$.
\end{thm}
\begin{defn}
The {\sf Aubry class} is defined by the element in the quotient space of $\cA(c)$ w.r.t. the following metric:
\[
d_c(\cdot,\cdot):\Big(\cA(c)\cap\{t=0\}\Big)\times \Big(\cA(c)\cap\{t=0\}\Big)\rightarrow\R
\]
via
\be\label{eq:aubry-clas}
 d_c(x,y)=h_c^\infty((x,0),(y,0))+h_c^\infty((y,0),(x,0))
\ee
for any $x,y\in\cA(c)\cap\{t=0\}$. Let's denote the quotient space by $\cA(c)/d_c$.
\end{defn}
\begin{thm}\cite{Mat3}
Restricted on $\cA(c)$, $d_c(\cdot,\cdot)\lesssim d(\cdot,\cdot)^2$, where $d(\cdot,\cdot)$ is the Euclid metric.
\end{thm}
From previous definitions we can easily see that 
\[
\widetilde{\mathcal{M}}(c)\subset\tilde{\mathcal{A}}(c)\subset\widetilde{\mathcal{N}}(c),
\]
which are all closed invariant set in $TM\times\T$.
\begin{rmk}
In \cite{Z3}, the author gave an example which shows previous 3 sets could be different. Besides, \cite{Z3} also showed that for generic Lagrangians, these 3 sets keep the same for generic $c\in H^1(M,\R)$.
\end{rmk}
\begin{lem}[Lemma 2.4 of \cite{CY1}] \label{lem:upper-semi}
The set-valued function 
\[
\Big\{c\in H^1(M,\R),|\cdot|\Big\}\longrightarrow\Big\{\wt\cN(c)\subset TM\times\T,d_H(\cdot,\cdot)\Big\}
\]
is upper semi-continuous w.r.t. previously given metric. Here $|\cdot|$ is the Euclid norm and $d_H(\cdot,\cdot)$ is the Hausdorff metric.
\end{lem}

\vspace{10pt}

\subsection{weak KAM solutions} Following the setting of Fathi in \cite{F}, we have:
\begin{defn}
A function $u:M\times\T\rightarrow\R$ is called {\sf dominated function} and is denoted by $u\prec L_c+\alpha(c)$, if $u(y,[t'])-u(x,[t])\leq F_c((x,[t]),(y,[t']))$ for all $(x,t,y,t')\in M\times\R\times M\times\R$. A curve $\gamma:(-\infty,s]\rightarrow M$ is called {\sf backward calibrated}, if 
\be\label{eq:back-cal}
u(\gamma(t'),[t'])-u(\gamma(t),[t])=A_c(\gamma)\big|_{[t,t']}
\ee
for all $t\leq t'\leq s$.
\end{defn}

\begin{defn}
A function  $u_c^-:M\times\T\rightarrow\R$ is called a {\sf backward weak KAM solution} of Hamiltonian $H(x,p,t)$, if 
\begin{itemize}
\item $u_c^-\prec L_c+\alpha(c)$;
\item $\forall (x,s)\in M\times\T$, there exists a calibrated curve $\gamma:(-\infty,s]\rightarrow M$ of $u_c^-(\cdot)$ ending with it.
\end{itemize}
\end{defn}
Here we exhibit a list of properties the weak KAM solutions possess, which are directly cited from \cite{C}:
\begin{prop}[Theorem 5, 9 of \cite{C}]\label{prop:tw-ma}
\begin{enumerate}
\item $u_c^-$ is a weak solution of $H(x,p,t)$, i.e.
\be\label{eq:ehj-gene}
\partial_t u_c^-+H(x,\partial_x u_c^-+c,t)=\alpha(c),\quad a.e.\ (x,t)\in M\times\T;
\ee
\item For each $c \in H^1(M,\R)$, $u_c^-$ is $K_c-$Lipshitz on $M\times\T$.
\item $u_c^-$ is differentiable at $\mathcal N(c)$;
\item for any $(z,[s])\in M\times\T$, the function 
\[
h_c^\infty((z,[s]),(\cdot,\cdot)):M\times\T\rightarrow\R
\]
 is a weak KAM solution.
\end{enumerate}
\end{prop}
\begin{defn}\cite{F}
A function $u:M\times\T\rightarrow\R$ is a {\sf viscous subsolution} of (\ref{eq:ehj-gene}), if for every $C^1-$function $\phi:M\times\T \rightarrow \R$ and every point $(x_0,\tau) \in M\times\T$ such that $u-\phi$ has a maximum at this point, and 
\[
\partial_t\phi(x_0,\tau)+H(x_0,\partial_x\phi(x_0,\tau)+c,\tau)\leq\alpha(c).
\]
A function $u:M\times\T\rightarrow\R$ is a {\sf viscous supersolution} of (\ref{eq:ehj-gene}), if for every $C^1-$function $\psi:M\times\T \rightarrow \R$ and every point $(x_0,\tau) \in M\times\T$ such that $u-\psi$ has a minimum at this point, and 
\[
\partial_t\phi(x_0,\tau)+H(x_0,\partial_x\phi(x_0,\tau)+c,\tau)\geq\alpha(c).
\]
A function $u:M\times\T\rightarrow\R$ is a {\sf viscous supersolution} of (\ref{eq:ehj-gene}), if it's both a viscous subsolution and a viscous supersolution.
\end{defn}

\begin{prop}[Proposition 3.12 of \cite{WY}]
The weak KAM solution $u_c^-$ is a viscous solution of system $H(x,p,t)$ and vice versa.
\end{prop}

\begin{defn}\cite{CS}
A function $f:\cU\subset\R^n\rightarrow\R$ is called {\sf semiconcave with linear modulus} (SCL for short), if there exists $C>0$ such that 
\be\label{eq:scl-defn}
f(x+h)+f(x-h)-2f(x)\leq C|h|^2
\ee
for all $x,h\in\R^n$ satisfying $B(x,h)\subset\cU$. Here the $C$ is called a {\sf semiconcavity constant} of $f$.
\end{defn}

\begin{defn}
For any SCL $f:\cU\rightarrow\R$, the {\sf super derivative set}  is defined by 
\[
D^+f(x):=\Big\{p\in\R^n\Big|\limsup_{y\rightarrow x}\frac{f(y)-f(x)-\langle p,y-x\rangle}{|y-x|}\leq 0\Big\},\quad \forall x\in\cU\subset\R^n.
\]
Moreover, $D^+f(x)\subset\R^n$ is a convex set of $\R^n$.
\end{defn}
\begin{rmk}
Due to (\ref{eq:scl-defn}), a SCL $f$ is differentiable at $x\in\cU$, iff $D^+f(x)$ is a singleton.
\end{rmk}

\begin{thm}[Theorem 6.4.1 of \cite{CS}]
$u_c^-(\cdot,t)$ is $SCL(M\times\T,\R)$.
\end{thm}

\begin{lem} \label{lem:ex-curve-semi}
For any extremal point $p\in ex(D^+u_c^-(x,t))$ with $(x,t)\in M\times\T$, there exists a backward semistatic curve $\gamma:(-\infty,t]\rightarrow M$ with $(\gamma(t),\dot\gamma(t),t)=(x,H_p(x,p+c,t),t)$, which calibrates the weak KAM solution $u_c^-(x,t)$.
\end{lem}
\begin{proof}
It's proved in Theorem 3.3.6  of \cite{CS}, that for any $t\in\T$ and $p\in ex(D^+u_c^-(x,t))$, there exists a sequence $\{x_n\}_{n\in\N}\subset M$ converges to $x$, such that $u_c^-$ is differentiable at $(x_n,t)$ and $\partial_x u_c^-(x_n,t)\rightarrow p$; Therefore, we can find a unique backwrad semistatic curve $\gamma_n$ ending with $(x_n,t)$, such that $\dot\gamma_n(t)=H_p(x_n,\partial_x u_c^-(x_n,t)+c,t)$. Since $\gamma_n$ is a calibrated curve of $u_c^-$, then $\{(\gamma_n(s),\dot\gamma_n(s),s)|s\in(-\infty,t]\}_{n\in\N}$ is compact in $TM\times\T$. So we can get an accumulating curve $\gamma_\infty$ of $\gamma_n$ by letting $n\rightarrow+\infty$. Due to (\ref{eq:back-cal}), we can easily see that $\gamma_\infty$ is a backward calibrated curve ending with $(x,t)$.
\end{proof}

\vspace{10pt}

\subsection{Variational conclusions of twist maps} Now we apply previous conclusions to the twist map, i.e. $M=\T$. Benefit from the low dimension, the system now inherits a bunch of fine properties, which are originally proved in the series of works of Mather in 1980s. As a direct citation, 
most of these conclusions can be found in \cite{Mat1,Mat2,Mat3,Mat5,B}.
\begin{prop} \label{prop:twist}
For Lagrangian $L(x,v,t)$ satisfying the standing assumptions, we have
\begin{enumerate}
\item {\sf (Sec.6.2 of \cite{B})}  $\alpha(c)$ is $C^1$ smooth;
\item {\sf (Prop.6 of \cite{Mat2})} $\beta(h)$ is strictly convex;
\item {\sf (Sec.2 of \cite{Mat1})} $\beta(h)$ is differentiable at $h\in\R\backslash\Q$, i.e. there exists a unique $c_h\in H^1(\T,\R)$ equals $\beta'(h)$; Besides, $\wt\cA(c_h)=\wt{\mathcal N}(c_h)$;
\item  {\sf (Sec.3 of \cite{Mat1})} If $\beta(h)$ is differentiable at $h\in\Q$, there exists a rotational invariant curve with rotation number $h$;
\item  {\sf (Sec.8 of \cite{Mat3})} If $\beta(h)$ is not differentiable at $p/q\in\Q$ in lowest terms, then $D^-\beta(p/q)=[c^-,c^+]$ is an interval; 
\begin{enumerate}
\item For any $c\in(c^-,c^+)$, $\wt\cA(c)=\wt\cM(c)$ consists of only $p/q-$periodic orbits; 
\item For $c^+$ (resp. $c^-$), $\wt\cA(c^+)$ (resp. $\wt\cA(c^-)$) additionally contains all the $\big(p/q\big)^+$ (resp. $\big(p/q\big)^-$) minimal heteroclinic orbits, and for any two points $(x,t), (y,t)\in\cA(c^+)$ (resp. $\in\cA(c^-)$), we have
\[
h_c^\infty((x,t),(y,t))+h_c^\infty((y,t),(x,t))=0.
\]
\end{enumerate}
\item {\sf (Sec.8 of \cite{Mat3})} When $D^-\beta(p/q)=[c^-,c^+]$, for any $c\in(c^-,c^+)$, any two points $(x,0),(y,0)\in\cA(c)$ are in different Aubry classes if and only if they are in different connected components of $\cA(c)\cap\{t=0\}$.
\end{enumerate}
\end{prop}
\begin{proof}
Here we display the precise citations where the readers could find the proof.
\end{proof}
\vspace{10pt}

\subsection{Generalized Characteristics of twist maps} For $M=\T$, the GC of (\ref{eq:gc}) possesses some fine properties as well. The first person revealed the propagations of GC is Dafermos \cite{D}, where he concerned certain Cauchy problem of Hamilton-Jacobi equation. Later, Cannarsa and Yu reproved these conclusions in an energy-optimal Language \cite{Ca}, based on the theory of SCL functions developped in \cite{CS}. We will adopt their approaches in this subsection.\\

Based on the semiconcavity of previous $u_c^-$, we rewrite (\ref{eq:gc}) here for convenience:
\be
\left\{
\begin{aligned}
&\dot x:=\frac{dx}{ds}= \text{co}\Big[ \partial_pH\Big(x(s),c+D^+u_c^-\big(x(s),s+\tau\big),s+\tau\Big)\Big],\\
&x(0)=x_0\in \T,\;\tau\in\T.
\end{aligned}
\right.
\ee
where co$(U)$ is the convex closure of any set $U\subset\R^n$. Recall that for any backward  semi-static curve $\gamma:(-\infty,t]\rightarrow \T$ with $(-\infty,t]$ being the maximal domain of  it, $u_c^-$ is always differentiable at $(\gamma(s),[s])$ for all $s<t$. Therefore, for any point $(\gamma(s),[s])$ with $s<t$, previous equation becomes
\[
\frac {dx}{d\xi}=\dot\gamma(s+\xi),\quad  x(0)=\gamma(s)
\]
with the maximal domain by $(-\infty,t-s)$. That's a {\sf regular GC} of the system $H(x,p,t)$. \\

For the non-differentiable point $(x_0,\tau)\in \T^2$ of $u_c^-$, previous Lemma \ref{lem:ex-curve-semi} implies the existence of several backward semi-static curves. That leads to an invalidity to define the backward GC ending with this point. However, the forward flow of (\ref{eq:gc}) is still achievable locally, which exists as a {\sf singular GC}. Namely, there exists a real number $\dt>0$, such that $\eta:[0,\dt)\rightarrow M$ is a solution of (\ref{eq:gc}) with $\eta(0)=x$, Moreover, the propagation of $\eta$ conforms to the following:
\begin{prop}\cite{Ca}\label{prop:exi}
For any $(x_0,\tau)\in \T^2$, there exists  a unique GC $\eta:[0,\sigma]\rightarrow \T$ of  (\ref{eq:gc}) starting from $(x_0,\tau)$, of which the right derivative $\dot\eta^+(s)$ exists for all $s\in[0,\sigma)$ and satisfies
\[
\dot\eta^+(s)=\dfrac{H\Big(\eta(s),p^-\big(\eta(s),s+\tau\big)+c,s+\tau\Big)-H\Big(\eta(s),p^+\big(\eta(s),s+\tau\big)+c,s+\tau\Big)}{p^-\big(\eta(s),s+\tau\big)-p^+\big(\eta(s),s+\tau\big)}
\]
if $(\eta(s),s+\tau)$ is a non-differentiable point of $u_c^-$ and
\[
\dot\eta^+(s)=H_p\Big(\eta(s),\partial_x u_c^-(\eta(s),s+\tau)+c,s+\tau\Big)
\]
if $(\eta(s),s+\tau)$ is a differentiable point of $u_c^-$. Here $\Big[p^-\big(\eta(s),s+\tau\big),p^+\big(\eta(s),s+\tau\big)\Big]=D^+u_c^-(\eta(s),s+\tau)\subset\R$.
 Moreover, $\dot\eta^+$ is right-continuous on $[0,\sigma)$.
\end{prop}
\begin{proof}
This result is essentially proved in the Theorem 3.11 of \cite{Ca}, where the proof is constructive and only a general semi-concavity of $u$ is needed. Here we just adapt it to our current setting, by adding the dependence of $c$, $\alpha(c)$. 
\end{proof}

Recall that for regular GC, the uniqueness holds. As for the singular GC, it holds as well:
\begin{prop}[Uniqueness]\label{prop:uni}
Starting from any non-differentiable point $(x_0,\tau)\in \T^2$ of $u_c^-$, (\ref{eq:gc}) has a unique solution. 
\end{prop}
\begin{proof}
Suppose $\xi:[0,\sigma]\rightarrow\T$ and $\eta:[0,\sigma]\rightarrow\T$ are two different GCs starting from $x_0$, then due to Proposition \ref{prop:exi}, we have 
\[
\dot\xi^+(s)\in\Big[H_p(\xi(s),p^+(\xi(s),s+\tau)+c,s+\tau),H_p(\xi(s),p^-(\xi(s),s+\tau)+c,s+\tau)\Big],
\]
\[
\dot\eta^+(s)\in\Big[H_p(\eta(s),p^+(\eta(s),s+\tau)+c,s+\tau),H_p(\eta(s),p^-(\eta(s),s+\tau)+c,s+\tau)\Big]
\]
for all $s\in[0,\sigma)$. Recall that $H_{pp}>0$ for everywhere of $\T^2$, and $H\in C^2(T\T\times\T,\R)$, then
\be
\frac d{ds}|\xi(s)-\eta(s)|^2&=&\big(\xi(s)-\eta(s)\big)\cdot\big(H_p(\xi(s),D^+u_c^-(\xi(s),s+\tau)+c,s+\tau)\nonumber\\
& &-H_p(\eta(s),D^+u_c^-(\eta(s),s+\tau)+c,s+\tau)\big)\nonumber\\
&\leq& K_c|\xi(s)-\eta(s)|^2\nonumber
\ee
for a.e. $s\in[0,\sigma]$. By the Gronwall's inequality we get $\xi\equiv\eta$.
\end{proof}
\vspace{20pt}

\section{$\frac{1}{2}-$H\"older regularity  of weak KAM solutions}\label{s3}

\vspace{20pt}

We devote this section to prove the Theorem \ref{thm:holder}. For this purpose, we could restrict the system to a section $\Sigma_0:=\{t=0\}$. Once we prove the $1/2-$H\"older continuity of $\{u_c^-(x,0)\}_{c\in H^1(\T,\R)}$, then Theorem \ref{thm:holder} will be proved since the section $\Sigma_0$ can be freely varied.\\

Next, we have to choose suitable weak KAM solutions. For any $c\in H^1(\T,\R)$, let's choose  $x_0\in\big(\cA(c)\cap\Sigma_0\big)$ being the closest point to $0$, and assume 
\be\label{eq:solution}
u_c^-(\cdot,t):=h_c^\infty((x_0,0),(\cdot,t)):\T\rightarrow\R
\ee
being the designated solution. For such a sequence $\{u_c^-(x,0)\}_{c\in H^1(\T,\R)}$, we can prove the following Lemmas:

\begin{lem}\label{lem:order}
For a fixed $x\in\T$, there are two backward semi-static curves $\gamma$, $\gamma'$ satisfying $\rho(\gamma)>\rho(\gamma')$, then 
\be\label{eq:order-ineq}
\dot\gamma^-(0)>\dot\gamma'^-(0),
\ee
where $\dot\gamma^-$ and $\dot\gamma'^-$ are the left derivatives respectively.
\end{lem}
\begin{proof}
As $\gamma$ (resp. $\gamma'$) is a backward semi-static curve,  so it has to be a minimizer of the following variational calculus:
\[
\cV(x,\rho(\gamma)):=\min_{\substack{\eta(0)=x\\\rho(\eta)=\rho(\gamma)}}\int_{-\infty}^0 L(\eta,\dot\eta,t)dt\quad\Big(\text{resp.\ }\cV(x,\rho(\gamma'))=\min_{\substack{\eta(0)=x\\\rho(\eta)=\rho(\gamma')}}\int_{-\infty}^0 L(\eta,\dot\eta,t)dt\Big).
\]
If $\dot\gamma^-(0)<\dot\gamma'^-(0)$, then due to $\rho(\gamma)>\rho(\gamma')$ there must be a $s\in(-\infty,0)$, such that $\gamma$ transversally intersects $\gamma'$ at time $s$. So we can find an open interval $(s-\sigma,s+\sigma)$ such that 
\be
& &h\big((\gamma(s-\sigma),s-\sigma),(\gamma(s+\sigma),s+\sigma)\big)+h\big((\gamma'(s-\sigma),s-\sigma),(\gamma'(s+\sigma),s+\sigma)\big)\nonumber\\
&<&h\big((\gamma(s-\sigma),s-\sigma),(\gamma'(s+\sigma),s+\sigma)\big)+h\big((\gamma'(s-\sigma),s-\sigma),(\gamma(s+\sigma),s+\sigma)\big)\nonumber
\ee
due to the {\sf Mather's Cross Lemma} (see Theorem 2 of \cite{Mat2}). That contradicts with the minimal property of semi-static curves, which instantly indicates (\ref{eq:order-ineq}).
\end{proof}
Due to the Legendre transformation, previous Lemma can be translated into the following:
\begin{cor}
For $c,c'\in H^1(\T,\R)$ satisfying $\alpha'(c)>\alpha'(c')$, and any $x\in\T$, we have 
\be
c+D^+u_c^-(x,0)>c'+D^+u_{c'}^-(x,0).
\ee
\end{cor}
\begin{proof}
Due to (\ref{eq:ineq-a-b}), the backward $c-$ (resp. $c'-$) semi-static curve $\gamma$ (resp. $\gamma'$) has a rotation number $\alpha'(c)$ (resp. $\alpha'(c')$). Due to Lemma \ref{lem:ex-curve-semi}, we know $\dot\gamma^-(0)=H_p(x,ex(D^+u_c^-(x,0))+c,0)$ (resp. $\dot\gamma'^-(0)=H_p(x,ex(D^+u_{c'}^-(x,0))+c',0)$). Recall $\partial^2_{pp}H>0$ everywhere due to the positive definiteness, then $\dot\gamma^-(0)>\dot\gamma'^-(0)$ is equivalent to $c+D^+u_c^-(x,0)>c'+D^+u_{c'}^-(x,0)$.
\end{proof}

Due to (5-a) of Proposition \ref{prop:twist}, we know that for any $p/q\in H^1(\T,\R)\cap\Q$ in lowest terms, there exists an interval $[c^-,c^+]$ which equals $D^-\beta(p/q)$. Moreover, for any $c\in(c^-,c^+)$, $\wt\cA(c)=\wt\cM(c)$ contains only $p/q-$periodic orbits and are both closed. Therefore,
\be
\T\backslash(\cA(c)\cap\Sigma_0)=\bigcup_{i\in\Lb}(r_i^-,r_i^+)
\ee
for an index set $\Lb\subset\N$. Within each {\sf gap} $(r_i^-,r_i^+)$, the following result holds:

\begin{lem}\label{lem:p/q-order}
For any $r\in(r_i^-,r_i^+)$, there exists a backward minimal curve $\xi$ (resp. $\xi'$) ending with $r$ and having the rotation symbol $(p/q)^+$ (resp. $(p/q)^-$), namely, $\xi$ (resp. $\xi'$) approximates to $\gamma_i^-$ (resp. $\gamma_i^+$) which is the $p/q-$minimal periodic curve ending with $r_i^-$ (resp. $r_i^+$). Then the left derivative obeys 
\be
\dot\xi'^-(0)<\dot\xi^-(0).
\ee
\end{lem}
\begin{proof}
The proof follows exactly the same procedure as Lemma \ref{lem:order}.
\end{proof}
Now we pick up a deck $[x_0,x_0+1]$ of the universal space $\R$, of which we can define two functions by 
\be\label{eq:upper-fun}
\wt u_{p/q+}(x):=\int_{x_0}^x\frac{\partial L}{\partial v}\big(y,\dot\xi_y^-(0),0\big)dy
\ee
\be\label{eq:lower-fun}
\wt u_{p/q-}(x):=\int_{x_0}^x\frac{\partial L}{\partial v}\big(y,\dot\xi'^-_y(0),0\big)dy
\ee
where $\xi_y$ (resp. $\xi'_y$) is the $(p/q)^+-$ (resp. $(p/q)^--$) backward minimal curve ending with $y$ if $y\in\T\backslash(\cA(c)\cap\Sigma_0)$, or $\xi_y=\xi'_y$ is the $(p/q)-$periodic minimal curve if $y\in\cA(c)\cap\Sigma_0$. Due to the uniform compactness of $(p/q)^\pm-$minimal curve, both $\wt u_{p/q\pm}(x)$ are Lipschitz continuous. Moreover, due to Lemma \ref{lem:p/q-order}, $\wt u_{p/q+}-\wt u_{p/q-}$ is nondecreasing on $[x_0,x_0+1]$. We will see that for all $c\in(c^-,c^+)$, the weak KAM solution $u_c^-$ formed in (\ref{eq:solution}) will be generated by these two functions:
\begin{lem}
Suppose $[c^-,c^+]=D^-\beta(p/q)$ is an interval, and $\{x_i\}_{i=0}^q$ is the $p/q-$periodic minimal configuration origining with $x_0\in(\cA([c^-,c^+])\cap\Sigma_0)$. For each $c\in(c^-,c^+)$, let's denote $\wt u_c^-(x):=u_c^-(x,0)+cx$ for $x\in[x_0, x_0+1]$, then
\begin{enumerate}
\item There exists a unique $x_i(c)\in(x_i,x_{i+m})$ for each $i\in\{0,1,2,\cdots,q-1\}$, where $m\in\{1,2,\cdots,q-1\}$ is choosen to make $x_{m}\in\R\slash[x_0,x_0+1]$ be closest point to $x_0$ of the configuration, such that 
\begin{enumerate}
\item $x_i(c)$ is monotonously increasing w.r.t. $c$;
\item $x_i(c)\rightarrow x_i$ as $c\rightarrow c^-$, and $x_i(c)\rightarrow x_{i+m}$ as $c\rightarrow c^+$;
\item $\wt u_c^-(x)=\wt u_{p/q+}(x)$ for $x\in[x_i,x_i(c)]$, and $\wt u_c^-(x)=\wt u_{p/q-}(x)$ for all $x\in[x_i(c),x_{i+m}]$.
\end{enumerate} 
\item For any $c<c'\in(c^-,c^+)$, $\wt u_c^-(x)\leq\wt u_{c'}^-(x)$ on $[x_0,x_0+1]$.
\end{enumerate}
\end{lem}
\begin{proof}
In the universal covering space $\R$ with a deck $[x_0,x_0+1]$, we consider the minimal configurations which are actually the intersectional points of backward semi-static curve with $\Sigma_0$. For any $x\in(x_i,x_{i+m})$ there are two backward minimal configurations $X_i(x)$, $Y_i(x)$ which approach $p/q-$periodic minimal configurations $X_i$ and $Y_i$, ending with $x_i$ and $x_{i+m}$ for $t=0$ respectively. We define
\[
\A_c^{p/q+}(x_i,x):=\lim_{k\rightarrow+\infty}\sum_{j=-1}^{-kq}h_c\big((X_i^j(x),j),(X_i^{j+1}(x),j+1)\big),\quad c\in[c^-,c^+]
\]
and $\A_c^{p/q-}(x_{i+m},x)$ can be defined in the similar way. Notice that $X_i^j(x)-X_i\rightarrow 0$ and $Y^j_i(x)-Y_i^j\rightarrow 0$ as $j\rightarrow-\infty$. The first obeservation is that:
\be
\A_c^{p/q-}(x_{i+m},x)-\A_c^{p/q+}(x_i,x)&=&\lim_{k\rightarrow+\infty}\sum_{j=-1}^{-kq}\Big[h\big((Y_i^j(x),j),(Y_i^{j+1}(x),j+1)\big)\nonumber\\
& &-h\big((X_i^j(x),j),(X_i^{j+1}(x),j+1)\big)\Big]+c(x_{i+m}-x_i),
\ee
which indicates this subtraction is {\sf strictly increasing} w.r.t. $c$ since $x_{i+m}-x_i>0$ for all $i\in\{0,1,\cdots,q-1\}$. Another observation is that for each $c\in(c^-,c^+)$, $\A_c^{p/q-}(x_{i+m},x)-\A_c^{p/q+}(x_i,x)$ is {\sf non-increasing} w.r.t. $x\in(x_i,x_{i+m})$. To prove this, we now pick $x<y$ being two points in $(x_i,x_{i+m})$, with the associated $X_i(x)$ and $X_i(y)$ (resp. $Y_i(x)$ and $Y_i(y)$) minimal configurations. We will analyse all the possible cases in the following:\vspace{5pt}

{\sf Case I.} Suppose there exists a gap $(r_k^-,r_k^+)\subset (x,y)$, if so, we have
\be
& &\Big[\A_c^{p/q-}(x_{i+m},x)-\A_c^{p/q+}(x_i,x)\Big]-\Big[\A_c^{p/q-}(x_{i+m},y)-\A_c^{p/q+}(x_i,y)\Big]\\
&=&\Big[\A_c^{p/q-}(x_{i+m},r_k^+)+\A_c^{p/q-}(r_k^+,r_k^-)+\A_c^{p/q-}(r_k^-,x)-\A_c^{p/q+}(x_i,x)\Big]\nonumber\\
& &-\Big[\A_c^{p/q-}(x_{i+m},y)-\A_c^{p/q+}(x_i,r_k^-)-\A_c^{p/q+}(r_k^-,r_k^+)-\A_c^{p/q+}(r_k^+,y)\Big]\nonumber\\
&=&\Big[\A_c^{p/q+}(x_i,r_k^-)+\A_c^{p/q+}(r_k^-,r_k^+)+\A_c^{p/q+}(r_k^+,y)\nonumber\\
& &+\A_c^{p/q-}(x_{i+m},r_k^+)+\A_c^{p/q-}(r_k^+,r_k^-)+\A_c^{p/q-}(r_k^-,x)\Big]\nonumber\\
& &-\Big[\A_c^{p/q-}(x_{i+m},y)+\A_c^{p/q+}(x_i,x)\Big]\nonumber
\ee
\be
&>&\Big[\underbrace{\A_c^{p/q+}(x_i,r_k^-)+\A_c^{p/q-}(r_k^-,x)}+\underbrace{\A_c^{p/q+}(r_k^+,y)+\A_c^{p/q-}(x_{i+m},r_k^+)}\Big]\nonumber\\
& &-\Big[\A_c^{p/q-}(x_{i+m},y)+\A_c^{p/q+}(x_i,x)\Big]\nonumber
\ee
\be
&\geq&\Big[\A_c^{p/q+}(x_i,x)+h_c^\infty((r_k^-,0),(r_k^-,0))+\A_c^{p/q-}(x_{i+m},y)+h_c^\infty((r_k^+,0),(r_k^+,0))\Big]\nonumber\\
& &-\Big[\A_c^{p/q-}(x_{i+m},y)+\A_c^{p/q+}(x_i,x)\Big]\nonumber\\
&=&0\nonumber
\ee
because $r_k^\pm\in\cA(c)\cap\Sigma_0$. So we proved $\A_c^{p/q-}(x_{i+m},x)-\A_c^{p/q+}(x_i,x)>\A_c^{p/q-}(x_{i+m},y)-\A_c^{p/q+}(x_i,y)$ for this case.\vspace{5pt}

{\sf Case II.} Suppose $x<y$ with $x\in(r_k^-,r_k^+)$, $r_k^+<y$ and $(x,y)$ containing no gap. For this case, 
\be
& &\Big[\A_c^{p/q-}(x_{i+m},x)-\A_c^{p/q+}(x_i,x)\Big]-\Big[\A_c^{p/q-}(x_{i+m},y)-\A_c^{p/q+}(x_i,y)\Big]\\
&=&\Big[\A_c^{p/q-}(x_{i+m},r_k^+)+\A_c^{p/q-}(r_k^+,x)-\A_c^{p/q+}(x_i,x)\Big]\nonumber\\
& &-\Big[\A_c^{p/q-}(x_{i+m},y)-\A_c^{p/q+}(x_i,r_k^-)-\A_c^{p/q+}(r_k^-,r_k^+)-\A_c^{p/q+}(r_k^+,y)\Big]\nonumber\\
&=&\Big[\A_c^{p/q+}(x_i,r_k^-)+\A_c^{p/q+}(r_k^+,y)+\A_c^{p/q-}(x_{i+m},r_k^+)+\underbrace{\A_c^{p/q+}(r_k^-,r_k^+)+\A_c^{p/q-}(r_k^+,x)}\Big]\nonumber
\ee
\be
& &-\Big[\A_c^{p/q-}(x_{i+m},y)+\A_c^{p/q+}(x_i,x)\Big]\nonumber\\
&>&\Big[\A_c^{p/q+}(x_i,r_k^-)+\A_c^{p/q+}(r_k^+,y)+\A_c^{p/q-}(x_{i+m},r_k^+)+h_c^\infty((r_k^+,0),(r_k^+,0))\nonumber
\ee
\be
& &+\A_c^{p/q+}(r_k^-,x)\Big]-\Big[\A_c^{p/q-}(x_{i+m},y)+\A_c^{p/q+}(x_i,x)\Big]\nonumber\\
&=&\underbrace{\A_c^{p/q+}(r_k^+,y)+\A_c^{p/q-}(x_{i+m},r_k^+)}-\A_c^{p/q-}(x_{i+m},y)\nonumber\\
&>&\A_c^{p/q-}(x_{i+m},y)+h_c^\infty((r_k^+,0),(r_k^+,0))-\A_c^{p/q-}(x_{i+m},y)\nonumber\\
&=&0.\nonumber
\ee
So we also proved $\A_c^{p/q-}(x_{i+m},x)-\A_c^{p/q+}(x_i,x)>\A_c^{p/q-}(x_{i+m},y)-\A_c^{p/q+}(x_i,y)$ for this case.\vspace{5pt}

{\sf Case III.} Tha last case is $[x,y]\subset\cA(c)\cap\Sigma_0$. If so, we can find a maximal interval $[\zeta^-,\zeta^+]\subset\cA(c)\cap\Sigma_0$ containing $[x,y]$. Actually, we can prove that 
\be\label{eq:case-3}
\A_c^{p/q-}(x_{i+m},x)-\A_c^{p/q+}(x_i,x)=\A_c^{p/q-}(x_{i+m},y)-\A_c^{p/q+}(x_i,y)
\ee
for any two $x<y\in[\zeta^-,\zeta^+]$. For this purpose, we have
\be
& &\Big[\A_c^{p/q-}(x_{i+m},x)-\A_c^{p/q+}(x_i,x)\Big]-\Big[\A_c^{p/q-}(x_{i+m},y)-\A_c^{p/q+}(x_i,y)\Big]\\
&=&\Big[\A_c^{p/q-}(x_{i+m},\zeta^+)+\A_c^{p/q-}(\zeta^+,x)-\A_c^{p/q+}(x_i,\zeta^-)-\A_c^{p/q+}(\zeta^-,x)\Big]\nonumber\\
& &-\Big[\A_c^{p/q-}(x_{i+m},\zeta^+)+\A_c^{p/q-}(\zeta^+,y)-\A_c^{p/q+}(x_i,\zeta^-)-\A_c^{p/q+}(\zeta^-,y)\Big]\nonumber\\
&=&\Big[\A_c^{p/q-}(\zeta^+,x)+\A_c^{p/q+}(\zeta^-,y)\Big]-\Big[\A_c^{p/q+}(\zeta^-,x)+\A_c^{p/q-}(\zeta^+,y)\Big].\nonumber
\ee
On one side, 
\be
& &\Big[\A_c^{p/q-}(\zeta^+,x)+\A_c^{p/q+}(\zeta^-,y)\Big]-\Big[\A_c^{p/q+}(\zeta^-,x)+\A_c^{p/q-}(\zeta^+,y)\Big]\\
&\leq&\Big[\A_c^{p/q-}(\zeta^+,y)+\A_c^{p/q-}(y,x)+\A_c^{p/q+}(\zeta^-,x)+\A_c^{p/q+}(x,y)\Big]\nonumber\\
& &-\Big[\A_c^{p/q+}(\zeta^-,x)+\A_c^{p/q-}(\zeta^+,y)\Big]\nonumber\\
&=& d_c(x,y)=0\nonumber
\ee
due to (\ref{eq:aubry-clas}) and (6) of Proposition \ref{prop:twist}. On the other side, 
\be
& &\Big[\A_c^{p/q-}(\zeta^+,x)+\A_c^{p/q+}(\zeta^-,y)\Big]-\Big[\A_c^{p/q+}(\zeta^-,x)+\A_c^{p/q-}(\zeta^+,y)\Big]\\
&\geq&\Big[\A_c^{p/q-}(\zeta^+,x)+\A_c^{p/q+}(\zeta^-,y)\Big]-\Big[\A_c^{p/q+}(\zeta^-,y)+\A_c^{p/q+}(y,x)\nonumber\\
& &+\A_c^{p/q-}(\zeta^+,x)+\A_c^{p/q-}(x,y)\Big]\nonumber\\
&=& -d_c(x,y)=0.\nonumber
\ee
This two inequalities imply (\ref{eq:case-3}) together.\\

Since we proved $\A_c^{p/q-}(x_{i+m},x)-\A_c^{p/q+}(x_i,x)$ is non-increasing of $x\in(x_i,x_{i+m})$, and we have 
\[
\lim_{x\rightarrow x_i}\A_c^{p/q-}(x_{i+m},x)-\A_c^{p/q+}(x_i,x)>0,
\]
\[
\lim_{x\rightarrow x_{i+m}}\A_c^{p/q-}(x_{i+m},x)-\A_c^{p/q+}(x_i,x)<0
\]
due to (5-b) of Proposition \ref{prop:twist}. Then due to the continuity of $\A_c^{p/q-}(x_{i+m},x)-\A_c^{p/q+}(x_i,x)$ there must be a $x_i(c)\in(x_i,x_{i+m})$ (for the interval $[\zeta_i^-(c),\zeta_i^+(c)]$ we pick the middle point) being the zero of  $\A_c^{p/q-}(x_{i+m},x)-\A_c^{p/q+}(x_i,x)$.

Due to the first observation, the function 
$c\rightarrow x_i(c)$ is strictly increasing. Besides, from the definition of $u_c^-$ in (\ref{eq:solution}), we know $\big(x,ex(D^+u_c^-(x,0),0\big)\in T\T\times\{0\}$ decides backward semi-static curve $\gamma_x^c$ which decides $(p/q)^+-$  (or $(p/q)^--$) minimal configuration. Due to previous analysis, we know $\gamma_x^c$ is $(p/q)^+-$type if $x\in(x_i,x_i(c))$ and $(p/q)^--$type if $x\in(x_i(c),x_{i+m})$. Accodingly, we have
\[
\begin{aligned}
&ex(D^+\wt u_c^-(x,0))=ex(D^+\wt u_{p/q+}(x)),\quad\forall x\in(x_i,x_i(c)),\\
&ex(D^+\wt u_c^-(x,0))=ex(D^+\wt u_{p/q-}(x)),\quad\forall x\in(x_i(c),x_{i+m}).
\end{aligned}
\]
From Lemma \ref{lem:p/q-order} we can instantly get conclusions (1-a), (1-c) and (2) of the current Lemma. At last, due to (5-b) of Proposition \ref{prop:twist}, we know that $x_i(c)\rightarrow x_{i+m}$ as $c\rightarrow c^+$ and $x_i(c)\rightarrow x_{i}$ as $c\rightarrow c^-$. So (1-b) of the  current Lemma will be proved.
\end{proof}

Notice that the function $x_i(c)$ in previous Lemma is just strictly increasing, which may not be absolutely continuous. That causes a difficulty to prove the H\"older continuity of $u_c^-$ about $c$ directly. So we have to introduce a substitution parameter of which $u_c^-$ is regularly dependent.
Benefit from previous Lemmas, now we introduce the following parameter
\be\label{eq:sigma}
\sigma(c)=\int_0^1\Big(\wh u_{c}^-(x)-\wh u_0^-(x)\Big)dx
\ee
which is monotonously increasing of $c$ and satisfies $\sigma(0)=0$. Here 
\[
\wh u_c^-(x):=\wt u_c^-(x)-\wt u_c^-(0),\quad \forall c\in H^1(\T,\R)
\]
 is a rectified weak KAM solution, and the purpose we doing so is to fixed $\wh u_c^-(0)\equiv0$ for all $c\in\R$. Therefore, we can unify all the $\{\wh u_c^-\}_{c\in\R}$ in the same deck $[0,1]$ of $\T$. 
\begin{proof}{\it of Theorem \ref{thm:holder}}:
As $c\rightarrow \sigma(c)$ is strictly increasing and continuous, so the inverse function $\sigma\rightarrow c(\sigma)$ is strictly increasing and continuous as well. Therefore, we make
\[
\|\wh u_{c(\sigma_1)}^--\wh u_{c(\sigma_2)}^-\|:=\max_{x\in[0,1]}|\wh u_{c(\sigma_1)}^-(x)-\wh u_{c(\sigma_2)}^-(x)|
\]
such that 
\be
|\sigma_1-\sigma_2|&=&\Big|\int_0^1\Big(\wh u_{c(\sigma_1)}^-(x)-\wh u_{c(\sigma_2)}^-(x)\Big)dx\Big|\nonumber\\
&\geq&\frac{1}{2|C_{\sigma_1}+C_{\sigma_2}|}\|\wh u_{c(\sigma_1)}^--\wh u_{c(\sigma_2)}^-\|^2
\ee
where $C_{\sigma_1}$ (resp. $C_{\sigma_2}$) is the Lipschitz constant of $\wh u_{c(\sigma_1)}^-$ (resp. $\wh u_{c(\sigma_2)}^-$). That leads to 
\be\label{eq:hull-solution}
\|\wh u_{c(\sigma_1)}^--\wh u_{c(\sigma_2)}^-\|\leq\sqrt{2|C_{\sigma_1}+C_{\sigma_2}|}|\sigma_1-\sigma_2|^{1/2}.
\ee
Notice that $\wh u_c^-(x)=u_c^-(x,0)-u_c^-(0,0)+cx$, where $u_c^-(x,0)$ is established by (\ref{eq:solution}). Then we get
the $1/2-$H\"older continuity of $\{\wh u_{c(\sigma)}^-(x,0)-\wh u_{c(\sigma)}^-(0,0)\}_{\sigma\in\R}$. So Theorem \ref{thm:holder} get proved.
\end{proof}

\vspace{20pt}

\section{Global existence and uniqueness of generalized characteristics}\label{s4}

\vspace{20pt}

For the equation (\ref{eq:gc}), we have already proved the existence of regular GC and singular GC in Sec.\ref{s2}. Actually, for any $(x,t)\in\T^2$, we can find an {\sf extended GC} formed by
\[
\wt\gamma(s):=\begin{pmatrix}
      \gamma(s)    \\
      t+s  
\end{pmatrix}\in\T^2
\]
where $\gamma(s)$ is a solution of (\ref{eq:gc}) in a suitable time interval. We will show the well definiteness of each extended GC by the following analysis.\\

\begin{proof}{\it of Theorem \ref{thm:g-c-t}}: 
The proof is twofold. First, if $(x,t)$ is a differentiable point of $u_c^-$, then due to lemma \ref{lem:ex-curve-semi}, there exists a unique backward semi-static curve $\gamma^-$ calibrates $u_c^-$ ending with $(x,t)$, which satisfies (\ref{eq:gc}) for $s\in(-\infty,0]$. Since $\gamma^-$ approximates $\cM(c)$ as $s\rightarrow-\infty$, then the rotation number $\rho(\gamma^-)=\rho(\cM(c))$, which implies $\rho(\gamma^-)=\alpha'(c)$ due to (1) of Proposition \ref{prop:twist}.\\

Second, if $(x,t)$ is a non-differentiable point of $u_c^-$, due to Proposition \ref{prop:exi}, there exists a singular GC $\gamma^+$ starting from $(x,t)$, which exists at least for time $s\in[0,\sigma]$. Following the idea of \cite{D}, we can prove that $(\gamma^+(s),t+s)$ is also non-differentiable point of $u_c^-$ for all $s\in[0,\sigma]$. If not, we can find a $\iota\in(0,\sigma]$, such that $u_c^-$ is differentiable at $(\gamma^+(\iota),t+\iota)$. Then once again we can find a backward semi-static curve $\gamma^-$ ending with $(\gamma^+(\iota),t+\iota)$, which satisfies (\ref{eq:gc}) for all $s\in(-\infty,0]$. However, due to Proposition \ref{prop:uni}, $\gamma^-\big|_{[-\iota,0]}$ has to equal $\gamma^+\big|_{[0,\iota]}$. That contradicts the non-differentiability of $u_c^-$ at the point $(x,t)$. So we proved the non-differentiability of $u_c^-$ for all $(\gamma^+(s),t+s)$ with $s\in[0,\sigma]$. On the other side, if $\sigma<+\infty$, then the non-differentiability implies that we can expand $\gamma^+$ for a bigger interval $[0,\sigma']$ with $\sigma'>\sigma$. Repeating previous procedure we conclude that $\gamma^+$ has to be a singular GC for all $s\in[0,+\infty)$.

Notice that $\gamma^+$ has no self-intersectional point, unless it's a periodic curve. For any cases, $\rho(\gamma^+)$ is well defined and equals to $\alpha'(c)$. The proof of this part is also twofold. First, if $\alpha'(c)=p/q\in\Q$, there must be an interval $[c^-,c^+]$ equal to $D^-\beta(p/q)$ due to (4) of Proposition \ref{prop:twist}. If $\rho(\gamma^+)\neq p/q$, then $\gamma^+$ has an intersectional point with $\cA(c)$. That contradicts the non-differentiability of $u_c^-$ along $\gamma^+$. So we proved $\rho(\gamma^+)=\alpha'(c)$ for the rational case.

If $\alpha'(c)=\om\in\R\backslash\Q$, then $\beta'(\om)=c$ due to (3) of Proposition \ref{prop:twist}. For this case, $\cM(c)=\cA(c)$ has to be a {\sf Denjoy minimal set}. So the extended $\wt\gamma^+$ has to lie in the gap of $\cA(c)$. As each gap of $\cA(c)$ has to shrink to zero as $s\rightarrow+\infty$ (Conclusion 4.6 of \cite{Ba}), so $\gamma^+$ has to approximate to $\cA(c)$ as $s\rightarrow+\infty$, so $\rho(\gamma^+)=\alpha'(c)$ for the irrational case.
\end{proof}


\vspace{20pt}

\section{$c-$equivalence of adjacent weak KAM solutions with singularities}\label{s5}

\vspace{20pt}

This section will be devoted to prove Theorem \ref{thm:s-t}. Throughout this section, we will restrict all the notions to the section $\Sigma_0:=\{t=0\}$. With the help of the conclusions proved in previous Sec. \ref{s3} and Sec. \ref{s4}, we figure out a modern way to explain the target theorem, which is more visualized to a high dimensional generalization.

\begin{lem}
For any $c\in(a_i,b_i)$ where $(a_i,b_i)$ is an instability interval, $\cN(c)\cap\Sigma_0$ is of zero homology class, namely, there exists an open neighborhood $U$ of $\cN(c)\cap\Sigma_0$, such that the inclusion map $H_1(U,\R)\hookrightarrow H_1(\T,\R)$ is trivial.
\end{lem}
\begin{proof}
Due to the definition of instability interval, we know that there must be singular GC of $u_c^-$ for $c\in(a_i,b_i)$. If so, we know that $\cN(c)\cap\Sigma_0$ has to be a strict closed subset of $\T$. Therefore, there exists an open neighborhood $U$ of $\cN(c)\cap\Sigma_0$, such that $U\subsetneq\T$. So $H_1(U,\R)$ is homologically trivial.
\end{proof}

\begin{lem}\label{lem:open}
For $c,c'\in(a_i,b_i)$ with $|c-c'|\ll 1$, there exists a unified open neighborhood $U$ containing both $\cN(c)\cap\Sigma_0$ and $\cN(c')\cap\Sigma_0$, which is homologically trivial.
\end{lem}
\begin{proof}
Due to Lemma \ref{lem:upper-semi}, we know that $c\rightarrow\cN(c)\cap\Sigma_0$ is upper semi-continuous. So for $c,c'\in(a_i,b_i)$ sufficiently close to each other, there must exist a unified open neighborhood $U$ containing both $\cN(c)\cap\Sigma_0$ and $\cN(c')\cap\Sigma_0$.
\end{proof}

The existence of a unified neighborhood $U$ gives us chance to conclude the following:
\begin{lem}\label{lem:class}
For $c,c'\in(a_i,b_i)$ with $|c-c'|\ll 1$, there exists a closed $1-$form $\mu(x)dx$ satisfying $supp(\mu)\cap U=\emptyset$ and $[\mu]=c'-c$.
\end{lem}
\begin{proof}
This conclusion is obvious since $U$ is homologically trivial.
\end{proof}

With the help of previous Lemma, now we establish a rectified variational calculus:
\begin{lem}[Lemma 2.3 of \cite{CY1}]\label{lem:c-c'}
For the modified Lagrangian
\[
L_{\eta,\mu,\rho}(x,v,t):=L(x,v,t)-\langle\eta(x)+\rho(t)\mu(x),v\rangle,\quad(x,v,t)\in T\T\times\R,
\]
where $[\eta]=c$, $[\mu]=c'-c$ and $\rho(t):\R\rightarrow\R$ is a smooth transitional function satisfying
\[
\rho(t)=\left\{
\begin{aligned}
 &1 \text{ for } t\in[\delta,\infty),\\
 & 0 \text{ for } t\in(-\infty,0],
\end{aligned}
\right.
\]
we can define an action function by
\be\label{eq:finite-pseudo}
h_{\eta,\mu,\rho}^{T_0,T_1}(m,m')=\inf_{\substack{\gamma(-T_0)=m\\ \gamma(T_1)=m'}}\int_{-T_0}^{T_1}L_{\eta,\mu,\rho}(\gamma(t),\dot{\gamma}(t),t)dt+T_0\alpha(c(s_0))+T_1\alpha(c(s))
\ee
for integers $T_0$, $T_1\in\N$ and
\be\label{pseudo}
h_{\eta,\mu,\rho}^\infty(m,m'):=\liminf_{T_0,T_1\rightarrow\infty}h_{\eta,\mu,\rho}^{T_0,T_1}(m,m'),\quad\quad\forall m,m'\in \T.
\ee
The function $h_{\eta,\mu,\rho}^\infty(\cdot,\cdot)$ is well defined as long as $c,c'\in(a_i,b_i)$ and $|c-c'|\ll1$. Moreover, 
If we denote by $\mathscr{C}_{\eta,\mu,\rho}$ the set of all the minimizers of (\ref{pseudo}), then any orbit $\gamma(t):\R\rightarrow \T$ in it conforms to the Euler-Lagrange equation 
\be\label{e-l}
\frac{d}{dt}L_v(\gamma(t),\dot{\gamma}(t),t)=L_x(\gamma(t),\dot{\gamma}(t),t),\quad\forall t\in\R,
\ee
and works as a heteroclinic orbit connecting $\mathcal{A}(c)$ and $\mathcal{A}(c')$. 
\end{lem}
\begin{proof}
As a direct citation of conclusions in \cite{CY1,CY2}, here we just give a sketch of the proof. Recall that $L_{\eta,\mu,\rho}$ is positively definite, so we can get the compactness of $\mathscr{C}_{\eta,\mu,\rho}$ which will be non-empty accordingly. Similar with Lemma \ref{lem:upper-semi}, $\mathscr{C}_{\eta,\mu,\rho}$ is upper semi-continuous w.r.t. $\big\{(\mu,\rho)\in H^1(\T,\R)\times C^\infty(\R,\R), |\cdot|\times \|\cdot\|\big\}$. Due to Lemma \ref{lem:open} and Lemma \ref{lem:class}, there exists $0<\dt<1$ suitably small, such that 
$\mathscr{C}_{\eta,\mu,\rho}\bigcap \Sigma_t\subset U$ as well, for any $t\in[0,\delta]$. This is because $\mathscr{C}_{\eta,0,0}=\mathcal{N}(c)$ and $\mathscr{C}_{\eta,\mu,\rho}$ is upper semi-continuous. That implies only for $t\in[0,\delta]$ we have $L_{\eta,\mu,\rho}\neq L$. However, $\forall\gamma\in\mathscr{C}_{\eta,\mu,\rho}$ should satisfies $\{\gamma(n)\}_{n=-\infty}^{+\infty}\cap$ supp$\mu(x)=\emptyset$, which implies for $t\in[0,\delta]$ conforms to the same Euler Lagrange equation as (\ref{e-l}).
\end{proof}
\vspace{5pt}

\noindent{\it Proof of Theorem \ref{thm:s-t}:}
It's easy to see that any two $c_1<c_2\in(a_i,b_i)$ are {\sf $c-$equivalent}, in the sense that we can input finitely many $\{c'_i\}_{i=1}^n$ contained in $[c_1,c_2]$, such that for any couple $(c_1,c'_1)$, $(c'_1,c'_2)$, $\cdots$, $(c'_i,c'_{i+1})$, $\cdots$, $(c'_{n-1},c'_n)$, $(c'_n,c_2)$, previous Lemma \ref{lem:c-c'} applies. Benefit from this property, we can find a so called {\sf transition chain} connect $\cA(c_1)$ to $\cA(c_2)$ and vice versa. Therefore, we can find a shadowing orbit which follows the interior part of the transition chain and visit suitably small neighborhoods of $\cA(c'_1)$ and $\cA(c'_n)$ in finite time. We can show that such a shadowing orbit is minimal for certain variational calculus formed like (\ref{eq:finite-pseudo}), then together with a variational calculus like (\ref{pseudo}) for $(c_1,c'_1)$ and $(c'_n,c_2)$, we can figure out a minimal orbit which taking $\cA(c_1)$ as the $\alpha-$limit set and $\cA(c_2)$ as the $\om-$limit set, vice versa.\qed

\begin{rmk}
In the model of convex billiards, \cite{Z1} has proved a similar result as Theorem \ref{thm:s-t}, even though the twist index is degenerate there.
\end{rmk}
\vspace{10pt}

\subsection{Outlook: from twist maps to high dimensional systems}

Previous discussions tell us that, the singularity of weak KAM solutions would never happen for an isolated $c\in H^1(M,\R)$, instead, it happens for a connected component of $c$, a so called {\sf Instability Region} $\cI_c\subset H^1(M,\R)$. For any $c\in\cI_c$, the singular GCs of $u_c^-$ form certain `topological obstruction', which will constraint the homology of $\cN(c)$. If so, we have $H_1(M,\cN(c),\R)\neq\emptyset$, and for any $\Dt c\in H^1(M,\R)$, which satisfies 
\[
\langle \Dt c,H_1(\cN(c),\R)\rangle=0
\]
in the sense of de Rham product. Therefore, for any $c,c'\in\cI_c$ with $|c-c'|\ll1$ and $ c-c'//\Dt c$, there should exist heteroclinic orbits connecting $\cA(c)$ and $\cA(c')$.

Notice that a `local surgery' with a rectified variational calculus formed like (\ref{eq:finite-pseudo}) and (\ref{pseudo}) is crucial to capture certain minimal heteroclinic orbits, since it constraints the changing of cohomology to a rather short time interval. If the Hamiltonian is autonomous, result in \cite{B} has shown such $c,c'$ has to lie on the same flat domain of $\alpha-$function, which indicates $\cA(c)\cap\cA(c')\neq\emptyset$ and then the heteroclinic connection becomes meaningless. On the other side, imitation of a similar variational principle as  (\ref{eq:finite-pseudo}) and (\ref{pseudo}) for the autonomous Hamiltonians is rather tricky and implicit \cite{C3}.\\

As a high dimensional extension, the uniqueness of singular GC and the well-definiteness for all $t\in[0,+\infty)$ should be the foremost difficulties we should overcome. Assuming mechanical systems seems to effectively ensure the uniqueness, and some evidence has been gotten in \cite{CCF}, which reveals certain homotopical equivalence between the singular GC and the projected Aubry set. However, if the maximal domain is finite, the singular GCs will not be able to form effective obstruction to the $\cN(c)$, which leads to a disability to construct local heteroclinic connection.

\begin{que}
For the Arnold type Hamiltonian 
\[
H_\eps(x,y,q,p,t)=\frac12 y^2+\frac12 p^2+(\cos q-1)\big(1+\eps V(x,t)\big),\quad\eps\ll1
\]
defined on $T^*\T\times T^*\T\times\T$, could we prove a similar conclusion as Theorem \ref{thm:g-c-t} ?
\end{que}
The practical meaning of this model is that for generic $V(x,t)\in C^2(\T^2,\R)$, we can construct diffusion orbits. Moreover, the normally hyperbolic invariant cylinder would assist us to constraint the topological state of singular GCs.

\end{document}